\documentclass[11pt]{amsart}

\usepackage[utf8]{inputenc}
\usepackage{mathrsfs}
\usepackage{tikz-cd}
\usepackage{amsmath}
\usepackage{amsthm}
\usepackage{dsfont}
\usepackage{amssymb}
\usepackage{url}
\usepackage{hyperref}
\usepackage{tabularx}

\newcommand{\opn}{\operatorname}

\newcommand{\Nat}{\mathbb{N}}
\newcommand{\Ghauss}{G\opn{-CGWH}}

\newcommand{\GCW}{G\opn{-CW}}
\newcommand{\Map}{\opn{Map}}
\newcommand{\Space}{\opn{Spc}}
\newcommand{\colim}{\opn{colim}}
\newcommand{\hocolim}{\opn{hocolim}}
\newcommand{\tot}{\opn{tot}}

\theoremstyle{thmstyleone}
\newtheorem{theorem}{Theorem}[section]
\newtheorem{proposition}[theorem]{Proposition}
\newtheorem{corollary}[theorem]{Corollary}
\newtheorem{lemma}[theorem]{Lemma}

\theoremstyle{thmstyletwo}
\newtheorem{example}[theorem]{Example}
\newtheorem{remark}[theorem]{Remark}

\theoremstyle{thmstylethree}
\newtheorem{definition}[theorem]{Definition}

\title{$\Ghauss$ spaces and their (homotopy) colimits}
\author{Aleksandar Miladinović}

\begin{document}

\begin{abstract}
In this paper we take a look at compactly generated weak Hausdorff spaces equipped with an action of a compact Lie group $G$ together with their colimits and homotopy colimits. In particular, we investigate relations between (homotopy) colimits and mapping spaces, and consequently, homotopy groups.
\end{abstract}

\maketitle
\tableofcontents



\section{Introduction}

Compactly generated spaces were popularized by Steenrod in \cite{Ste67} and later by McCord in \cite{McC69} who added the weak Hausdorff condition. Steenrod referred to the compactly generated spaces as a \emph{convenient category of spaces}, which is further supported by the fact that much of the literature today on equivariant and stable homotopy theory uses compactly generated spaces. These spaces have been extensively studied and provide a well-behaved setting for many topological constructions.

This paper is motivated by the study of homotopy colimits in the context of compactly generated $G$-spaces, where $G$ is a compact Lie group. In particular, we focus on the relations between homotopy colimits, equivariant mapping spaces, and homotopy groups. These results appear to play an important role in understanding the $\infty$-category of $G$-spaces, yet detailed references are often absent or scattered in the existing literature. The goal of this paper is to present a coherent collection of such results, intended to serve as a useful resource for future work in equivariant homotopy theory and related areas.

We will open this paper with a section on model categories where our main focus will be on the functorial presentation of the homotopy colimit. The following section will be dedicated to the compactly generated weak Hausdorff spaces equipped with an action of a compact Lie group $G$ ($\Ghauss$ spaces for short). In that section, we will focus on the interaction of the mapping spaces, and consequently homotopy groups with colimits and homotopy colimits. Let us recall that, if $X$ is a colimit of a sequence of closed inclusions of spaces $X_1\rightarrow X_2\rightarrow ...$ then the natural map
\[
\underset{n\in\Nat}{\colim}\, \pi_k(X_n)\rightarrow \pi_k(X)
\]
is an isomorphism (see \cite[Ch. 8, section 4]{MCC}). We will show that this isomorphism can be upgraded. In particular, we will show that, given a sequence of closed inclusions of $\Ghauss$-spaces $X_1\rightarrow X_2\rightarrow ...$ and a closed subgroup $H\leq G$, the natural map
\[
\underset{n\in\Nat}{\colim}\, \pi^H_k(X_n) \rightarrow \pi^H_k(\underset{n\in\Nat}{\hocolim}\, X_n)
\]
is an isomorphism.

\section{Model categories}

The theory of model categories was developed by Quillen in \cite{Qui67} and \cite{Qui69}. Formally, a model category is an ordinary category with three specified classes of morphisms: \emph{weak equivalences}, \emph{fibrations} and \emph{cofibrations}, which are subject to a certain set of axioms. The main advantage of model categories is that they provide us with decent machinery with which we can operate in homotopy theory. Some good references are \cite{Hov}, \cite{Hir}, \cite{DS95} and \cite{Gin}, but there many others.

\begin{definition}
Let $M$ be an ordinary category and let $q:A\rightarrow B$ and $f:X\rightarrow Y$ be two morphisms in $M$. The morphism $q$ is called a \emph{retract} of $f$ if there is a commutative diagram
\begin{center}
\begin{tikzcd} [row sep=2em, column sep=2em]
A \arrow[r]\arrow[rr, bend left, "id_A"] \arrow[d, "q"] & X \arrow[d,"f"]\arrow[r] & A \arrow[d,"q"]\\
B \arrow[r]\arrow[rr,bend right, "id_B"] & Y \arrow[r] & B
\end{tikzcd}
\end{center}
\end{definition}

\begin{remark}
If we take $M$ to be a category of topological spaces and $q=id_A$, $f=id_X$ with $A\subseteq X$, then the notion of $id_A$ being a retract of $id_X$ corresponds to the notion of space $A$ being a retract of a topological space $X$.
\end{remark}

\begin{definition}
A \emph{model category} is a category $M$ equipped with three classes of morphisms:
\begin{itemize}
\item class of \emph{weak equivalences} $\mathcal{W}$ (which are denoted with  $\overset{\sim}{\rightarrow}$),
\item class of \emph{fibrations} $\mathcal{F}$ (denoted with $\twoheadrightarrow$),
\item class of \emph{cofibrations} $\mathcal{C}$ (denoted with $\rightarrowtail$),
\end{itemize}
which satisfy the following $5$ axioms:
\begin{enumerate}
\item $M$ is cocomplete;
\item ($2$ out of $3$ property) For every commutative diagram
\begin{center}
\begin{tikzcd} [row sep=2em, column sep=2em]
 & B \arrow[rd] & \\
A \arrow[rr]\arrow[ru] & & C
\end{tikzcd}
\end{center}
if two of the three morphisms are weak equivalences then so is the third;
\item If $q$ is a retract of $f$ such that $f\in\mathcal{W}$ (resp. $f\in\mathcal{F}$, $f\in\mathcal{C}$), then $q\in\mathcal{W}$ (resp. $q\in\mathcal{F}$, $q\in\mathcal{C}$);
\item For every commutative diagram
\begin{center}
\begin{tikzcd} [row sep=2em, column sep=2em]
 A \arrow[d,rightarrowtail,"i"] \arrow[r] & B \arrow[d, twoheadrightarrow,"p"] \\
 C \arrow[r]\arrow[ur,dashed] & D
\end{tikzcd}
\end{center}
such that $i\in\mathcal{C}$ and $p\in\mathcal{F}$ the dashed lift exists if $i$ or $p$ is in addition a weak equivalence;
\item Every morphism $f:X\rightarrow Y$ admits two natural factorizations
\begin{align*}
X & \overset{\sim}{\rightarrowtail} P_f \twoheadrightarrow Y \text{ , and}\\
X & \rightarrowtail C_f \overset{\sim}{\twoheadrightarrow} Y
\end{align*}
\end{enumerate}
We will write $(M,\mathcal{W},\mathcal{F},\mathcal{C})$ for a model category $M$ to indicate the classes of weak equivalences, fibrations and cofibrations. Additionally, when the model structure is established, or known, we will simply write $M$, omitting the classes of morphisms.
\end{definition}
By axiom $1$ every model category $M$ has an initial and a final object since these objects correspond to the limit and colimit of empty diagrams respectively. For future reference, let us denote with $0$ the initial and with $\{\ast\}$ the terminal object of $M$.
\begin{definition}
Let $X\in M$ where $M$ is a model category. We say that $X$ is
\begin{itemize}
\item \emph{fibrant} if we have $X\twoheadrightarrow \{\ast\}$;
\item \emph{cofibrant} if we have $0\rightarrowtail X$.
\end{itemize}
\end{definition}

\begin{definition}
For every $X\in M$ consider the maps $0\rightarrow X$ and $X\rightarrow \{\ast\}$. Axiom $5$ gives us a factorization
\begin{align*}
X & \overset{\sim}{\rightarrowtail} P_X \twoheadrightarrow \{\ast\} \text{ , and}\\
0 & \rightarrowtail C_X \overset{\sim}{\twoheadrightarrow} X \, .
\end{align*}
We call $P_X$ (resp. $C_X$) the \emph{fibrant} (resp. \emph{cofibrant}) replacement of $X$.
\end{definition}

\begin{remark}
Given a model category $M$, we will denote with $M_f$ (resp. $M_c$) the full subcategory spanned by fibrant (resp. cofibrant) objects. By axiom $5$ we have functors
\begin{align*}
R: & M\rightarrow M_f\\
L: & M\rightarrow M_c
\end{align*}
such that $R(X) = P_X$ and $L(X) = C_X$.
\end{remark}

\begin{example}
\label{quillenms}
\textbf{Quillen model category:} Quillen gave a proof that the category of topological spaces $\operatorname{Top}$ can be endowed with the structure of a model category with:
\begin{itemize}
\item the class of weak equivalences being \emph{weak homotopy equivalences},
\item fibrations being \emph{Serre fibrations}, and
\item cofibrations being retracts of general cellular inclusions.
\end{itemize}
It is worth noting that in this structure every object is fibrant.
\end{example}

\begin{definition}
Let $M$ be a model category with $\mathcal{W}$ the class of weak equivalences. We define the \emph{homotopy category} of $M$ to be the localization of $M$ at the class of maps $\mathcal{W}$ and we denote it with
\[
Ho(M) := M[\mathcal{W}^{-1}]
\]
Note that we also have the projection functor $\pi: M\rightarrow Ho(M)$.
\end{definition}

\begin{lemma}
Let $M$ be a model category. The canonical inclusions $M_f\hookrightarrow M$ and $M_c\hookrightarrow M$ induce equivalences of homotopy categories
\begin{align*}
Ho(M_f) & \xrightarrow{\simeq} Ho(M)\\
Ho(M_c) & \xrightarrow{\simeq} Ho(M)\\
\end{align*}
\end{lemma}

Let us assume that, from now on, when choosing a colimit in a category, we are doing so in a functorial way. In such setting, we may define the homotopy colimit functorially as well, with such functor being \emph{derived} from the colimit functor. For this we will define what derived functors in a model category are.

\begin{definition}
Let $C$ and $D$ be two model categories ad let $F:C\rightarrow D$ and $G:D\rightarrow C$ be two functors. We define the following.
\begin{itemize}
\item \emph{Left derived} functor of $F$, $\mathbb{L}F:Ho(C)\rightarrow D$ to be the Left Kan extension of $F$ along $\pi_C : C\rightarrow Ho(C)$,
\item \emph{Right derived} functor of $G$, $\mathbb{R}G:Ho(D)\rightarrow C$ to be the Right Kan extension of $G$ along $\pi_D : D\rightarrow Ho(D)$,
\item \emph{Total} left derived functor of $F$, $\mathbb{L}^{\tot} F$ as the left derived functor of the composition $C\xrightarrow{F}D\xrightarrow{\pi_D}Ho(D)$,
\item \emph{Total} right derived functor of $G$, $\mathbb{R}^{\tot} G$ as the right derived functor of the composition $D\xrightarrow{G}C\xrightarrow{\pi_C}Ho(C)$.
\end{itemize}
\end{definition}

\begin{example}
\label{defhocolim}
The most important examples of total derived functors are \emph{homotopy} colimits.

Note that, given a category $K$ and a category $C$ we have a category of $K$-indexed diagrams $C^K$. Moreover, we have an adjunction $\colim: C^K \rightleftarrows C : c$ where $\colim$ is a functor which assigns to every $K$-diagram $X:K\rightarrow C$ its colimit and where $c$ is a constant functor.

If $C$ is a model category we can equip $C^K$ with the \emph{projective model structure} in which the weak equivalences (resp. fibrations) are those natural transformations $\tau: X\rightarrow Y$ (where $X,Y:K\rightarrow C$) such that $\tau_k : X(k)\rightarrow Y(k)$ is a weak equivalence (resp. fibration) for every $k\in K$. The cofibrations would be those natural transformations that have the left lifting property (LLP) with respect to all trivial fibrations (i.e. fibrations that are also weak equivalences).

Considering $\colim$ functor as a functor between two model categories we will denote with
\[
\hocolim := \mathbb{L}^{\tot} \colim
\]
the homotopy colimit functor. We can make an analogous definition for homotopy limits.
\end{example}

\begin{example}
\label{comphocolim}
Let $M$ be a model category, $K$ an ordinary category and $X:K\rightarrow M$ a $K$-diagram. As we have defined in \ref{defhocolim}, $\underset{K}{\hocolim}\, X = \mathbb{L}^{\tot} \underset{K}{\colim}\, X$. By the definition of derived functors $\underset{K}{\hocolim}\, X = \underset{K}{\colim}\, LX$ where $LX$ is the cofibrant replacement of the $K$-diagram $X$ in the projective model structure on $M^K$ (see, for example the proof of \cite[Proposition 2.5.13]{Gin}).

For example, let
\[
X_1\rightarrow X_2 \rightarrow ... \rightarrow X_n \rightarrow ...
\]
be a sequence in $M$. By inspection, the cofibrant replacement of given sequence is given by the sequence that fits into the following commutative diagram
\begin{center}
\begin{tikzcd} [row sep=2em, column sep=2em]
X_1 \arrow[r] & X_2 \arrow[r] & ... \arrow[r] & X_n \arrow[r] & ... \\
CX_1 \arrow[r]\arrow[u] & CX_2 \arrow[r]\arrow[u] & ... \arrow[r] & CX_n \arrow[r]\arrow[u] & ...
\end{tikzcd}
\end{center}
where $CX_i\rightarrow X_i$ is a weak equivalence and $CX_i \rightarrow CX_{i+1}$ is a cofibration for every $i\in\Nat$, thus
\[
\underset{i\in\mathbb{N}}{\hocolim}\, X_i = \underset{i\in\mathbb{N}}{\colim}\, CX_i \, .
\]
\end{example}

\section{$\Ghauss$ spaces}

\begin{definition}
Let us denote with $\Ghauss$ the category whose objects are compactly generated weak Hausdorff spaces equipped with an action of a compact Lie group $G$, and whose morphisms are continuous (not necessarily equivariant) maps.
\end{definition}

\begin{remark}
One may notice that when speaking of compactly generated weak Hausdorff spaces, one often omits the latter part, writing only \emph{compactly generated} and assuming that the spaces are weak Hausdorff by default (\cite{Rezk}, \cite[Appendix A]{Swe}), hence, when speaking of elements of the category $\Ghauss$ we will often write $\Ghauss$ space, or just compactly generated $G$-space.
\end{remark}

As noted in \cite[Appendix B]{Swe}, the category of $\Ghauss$ spaces is (co)complete. Using \cite[Proposition B.2 (i)]{Swe}, the orbit space $G/H$ is compactly generated, where $H\leq G$ is a closed subgroup, which gives us the following result.

\begin{proposition}
Every $\GCW$-complex is compactly generated.
\end{proposition}
\begin{proof}
Let us first prove this claim for finite $\GCW$-complexes. We can do this by induction on the skeleton. Let $X$ be a finite $\GCW$-complex. The $0$-skeleton, $X_0$ is simply a disjoint union of orbit spaces which are compactly generated as mentioned above. Let us assume that $n-1$-skeleton of $X$, $X_{n-1}$, is compactly generated. Now the $n$-skeleton is obtained via the pushout
\begin{center}
\begin{tikzcd} [row sep=2em, column sep=2em]
\underset{j}{\coprod} G/H_j \times S^{n-1} \arrow[d]\arrow{r} & X_{n-1} \arrow[d]\\
\underset{j}{\coprod} G/H_j \times D^n \arrow[r] & X_n
\end{tikzcd}
\end{center}
Note that this diagram is a pushout square of compactly generated spaces, since $S^{n-1}$ and $D^n$ are compact, hence $X_n$ is compactly generated as well.

Finally, let $X$ be a $\GCW$ complex. Then $X$ can be presented as the colimit of a sequence of finite $\GCW$-complexes $X_0\rightarrow X_1\rightarrow X_2\rightarrow ...$ where each $X_k\rightarrow X_{k+1}$ is a closed inclusion. Since all of the elements of the sequence are compactly generated, so will $X$ be.
\end{proof}

The category of $\Ghauss$ spaces can be endowed with a model structure, which we call the \emph{standard model structure for $G$-spaces} (\cite[Proposition B.7]{Swe}), where:

\begin{itemize}
\item \textbf{weak equivalences} are weak homotopy equivalences of $G$-spaces, that is, $G$-maps $f:X\rightarrow Y$ such that $f^H:X^H\rightarrow Y^H$ is a weak homotopy equivalence for every closed subgroup $H\leq G$;
\item \textbf{fibrations} are Serre fibrations of $G$-spaces, that is, $G$-maps $f:X\rightarrow Y$ such that $f^H:X^H\rightarrow Y^H$ is a Serre fibration for every closed subgroup $H\leq G$;
\item \textbf{cofibrations} are maps which satisfy the left lifting property with respect to acyclic fibrations (i.e. fibrations that are also weak equivalences).
\end{itemize}
Note that, similar as in \ref{quillenms} every $\Ghauss$ space is fibrant in the standard model structure. Additionally, by \cite[Ch. I, Theorem 3.1]{CBMS}, $\GCW$-complexes are fibrant objects. Moreover, analogous to the proofs of \cite[2.4.5, 2.4.6]{Hov} we can conclude:
\begin{proposition}
\label{gincl}
Every cofibration in the standard model structure on $\Ghauss$ spaces is a closed $G$-inclusion.
\end{proposition}

\begin{remark}
\label{lrem21}
When $K$ is compact and $X$ is compactly generated, the mapping space $\Map_{\Space}(K,X)$ of all continuous maps is compactly generated (see \cite[Proposition 8.2]{Rezk}). If we take spaces $K$ and $X$ to be equipped with an action of a compact Lie group $G$, we will denote with $\Map(K,X)$ the $\Ghauss$ space of all continuous maps with $G$ acting by conjugation. Additionally, given a $\Ghauss$ space $X$ and a closed subgroup $H\leq G$, the space $X^H$ is a closed subspace of $X$ and hence compactly generated (\cite[Proposition A.5, Appendix B]{Swe}), meaning that the space $(\Map(K,X))^H := \Map_H(K,X)$ of all $H$-equivariant continuous maps is compactly generated.
\end{remark}

\begin{proposition}
\label{tfact}
Let
\[
X_0 \rightarrow X_1 \rightarrow ...\rightarrow X_n \rightarrow ..
\]
be a sequence of $\Ghauss$ spaces where the maps $X_i \rightarrow X_{i+1}$ are closed $G$-inclusions. Denote with $X$ the colimit of the sequence. Then for every compact $G$-space $K$ we have a bijection
\begin{equation}
\label{bcolim}
\tau : \underset{n\in\Nat}{\colim}\, \Map(K,X_n)\rightarrow \Map(K, X)
\end{equation}
which is $G$-equivariant. Moreover, for every closed subgroup $H\leq G$, we have a bijection
\[
\tau^H: \underset{n\in\Nat}{\colim}\, \Map_H(K,X_n)\rightarrow \Map_H(K, X)
\]
\end{proposition}
\begin{proof}
Let $[f']\in \underset{n\in\Nat}{\colim}\, \Map(K,X_n)$ be an element of the colimit represented by a map $f'\in \Map(K,X_n)$ for some $n\in\Nat$. By \cite[Proposition A.14]{Swe} we have a closed inclusion $i_n: X_n\rightarrow X$ which will also be a $G$-map since all of the maps in the sequence are $G$-maps, hence, $\tau$ can be viewed as a map sending an element $[f']$ to the map $i_n\circ f' \in \Map(K,X)$. Additionally, since every map $f\in \Map(K,X)$ factors through some $X_n$ (see \cite[10.14]{Rezk}) i.e. $f=i_n\circ f'$ for some $f'\in\Map(K,X_n)$, the element of $\underset{n\in\Nat}{\colim}\, \Map(K,X_n)$ represented by $f'$ maps to $f$ by $\tau$, hence $\tau$ is surjective. Furthermore, let $[f_1],[f_2]$ be two non-equal elements of $\underset{n\in\Nat}{\colim}\, \Map(K,X_n)$. For large enough $n\in\Nat$ we may represent these elements by maps $f_1,f_2:\Map(K,X_n)$. Then these elements are sent by $\tau$ to the maps $i_n\circ f_1$ and $i_n\circ f_2$. Since $f_1\neq f_2$, then $i_n\circ f_1\neq i_n\circ f_2$, hence $\tau$ is injective, and a bijection.

What is left to check is that $\tau$ is a $G$-map. Again, let $[f]$ be an element of $\underset{n\in\Nat}{\colim}\, \Map(K,X_n)$ represented by a map $f\in \Map(K,X_n)$, and let $g\in G$. Recall that $\tau([f]) = i_n\circ f$. Since the induced maps $\Map(K,X_i)\rightarrow\Map(K,X_{i+1})$ are closed $G$-inclusions, the element $g[f]$ corresponds to the element $[gf]$. Since all of the maps $X_k\rightarrow X_{k+1}$ are closed $G$-inclusion, so will $i_n$ be, hence $\tau(g[f]) = i_n\circ (gf)=g(i_n\circ f)= g\tau([f])$ i.e. $\tau$ is $G$-equivariant.
\end{proof}

\begin{remark}
Since we are working in the category $\Ghauss$ we ought to prove that $\tau$ is in fact a continuous map in \ref{bcolim}, but this part has been omitted as to not stray too far away from the topic of $\Ghauss$ spaces. Nonetheless, proving that $\tau$ is continuous can be left as an exercise to the reader.
\end{remark}

Before moving forward let us recall the definition of the homotopy groups of equivariant spaces.
\begin{definition}
Let $H\leq G$ be a closed subgroup of a compact Lie group $G$, let $X$ be a $G$-space and let $x\in X$ be a $G$-fixed point. Then we define
\[
\pi^H_n(X,x) := \pi_n(X^H,x)
\]
for $n\in\Nat$.
\end{definition}
As mentioned in \ref{lrem21}, given a $\Ghauss$ space $X$ and a closed subgroup $H\leq G$, the space $X^H$ is a closed subspace of $X$ and hence compactly generated. Then the closed inclusion of $G$-spaces $X\rightarrow Y$ gives us a closed inclusion $X^H\rightarrow Y^H$. Additionally, taking $H$-fixed points commutes with filtered colimits (and, in particular, sequential colimits) along $G$-maps that are closed inclusions, \cite[Proposition B.1 (ii)]{Swe}, hence we obtain the following lemma.

\begin{lemma}
\label{lc24}
Let $H\leq G$ be a closed subgroup. Then for a sequence
\[
X_0 \rightarrow X_1 \rightarrow ...\rightarrow X_n \rightarrow ..
\]
of $\Ghauss$ spaces where the maps $X_i \rightarrow X_{i+1}$ are closed $G$-inclusions and every $G$-fixed point $x\in X_0$, the map
\[
\underset{n\in\Nat}{\colim}\, \pi^H_k(X_n, x) \rightarrow \pi^H_k(\underset{n\in\Nat}{\colim}\, X_n, x)
\]
is an isomorphism for every $k\geq 0$.
\end{lemma}
\begin{proof}
Let $X$ be the colimit of the upper sequence. Since taking $H$-fixed points commutes with sequential colimits, we have $\underset{n\in\Nat}{\colim}\, X_n^H = X^H$. The statement now follows from the result in \cite[Ch. 8, section 4]{MCC} stating that $\underset{n\in\Nat}{\colim}\, \pi_k(X^H_n, x) \rightarrow \pi_k(X^H, x)$ is an isomorphism.
\end{proof}

\begin{proposition}
\label{lc25}
Let
\[
X_0 \rightarrow X_1 \rightarrow ...\rightarrow X_n \rightarrow ..
\]
be a sequence of $\Ghauss$ spaces where the maps $X_i \rightarrow X_{i+1}$ are closed $G$-inclusions. Then the colimit $\underset{n\in\Nat}{\colim}\, X_n$ in the category of $\Ghauss$ spaces is also the homotopy colimit with respect to the standard model structure on $G$-spaces.
\end{proposition}
\begin{proof}
The homotopy colimit of the upper sequence is computed by taking the cofibrant replacement of the sequence in the projective model structure and then computing the colimit. In particular, we arrive to the familiar commutative diagram
\begin{center}
\begin{tikzcd} [row sep=2em, column sep=2em]
X_1 \arrow[r] & X_2 \arrow[r] & ... \arrow[r] & X_n \arrow[r] & ... \\
CX_1 \arrow[r]\arrow[u,"f_1"] & CX_2 \arrow[r]\arrow[u,"f_2"] & ... \arrow[r] & CX_n \arrow[r]\arrow[u,"f_n"] & ...
\end{tikzcd}
\end{center}
where the maps $f_i:CX_i\rightarrow X_i$ are $G$-weak homotopy equivalences and where $CX_i\rightarrow CX_{i+1}$ are closed $G$-inclusions. Let $X$ and $CX$ be the colimits of the upper and the lower sequence respectively. Additionally, we have an induced map $f:CX\rightarrow X$ and we want to show that this map is a $G$-map. Let $x\in CX$, $y\in X$ such that $y=f(x)$ and let $g\in G$. Additionally, let $n\in\Nat$ be great enough so that $x$, $gx$, $y$, $gy$ are represented by the images of some elements $x', gx'\in CX_n$ and $y',gy'\in X_n$. Since $f(x)=y$ by assumption, we may deduce $f_n(x')=y'$. Since the upper diagram is a commutative diagram of $G$-maps, we have $f_n(gx')=gf_n(x')=gy'$ i.e. $f(gx)=gy=gf(x)$. In other words, $f$ is a $G$-equivariant map.

In order to show further that $f$ is a $G$-weak homotopy equivalence, let us pass to the diagram of $H$-fixed points, with $H\leq G$ a closed subgroup. We arrive to the following diagram
\begin{center}
\begin{tikzcd} [row sep=2em, column sep=2em]
X_1^H \arrow[r] & X_2^H \arrow[r] & ... \arrow[r] & X_n^H \arrow[r] & ... \\
CX_1^H \arrow[r]\arrow[u,"f_1^H"] & CX_2^H \arrow[r]\arrow[u,"f_2^H"] & ... \arrow[r] & CX_n^H \arrow[r]\arrow[u,"f_n^H"] & ...
\end{tikzcd}
\end{center}
where both sequences are sequences of closed inclusions of compactly generated spaces, and where the maps $f_i^H$ are all weak homotopy equivalences. Since taking fixed points commutes with filtered colimits, and in particular, sequential colimits \cite[Proposition B.1]{Swe}, the colimits of the upper and lower sequence are $X^H$ and $CX^H$ respectively, and the induced map on colimits is exactly $f^H$. Then, by \cite[Proposition A.17]{Swe} the map $f^H:CX^H\rightarrow X^H$ is a weak homotopy equivalence as well, hence, $f$ is a $G$-weak homotopy equivalence. Homotopy colimits are unique up to an isomorphism in the homotopy category of the model category, hence, the proof is finished since $f$ is a weak equivalence in the standard model structure.
\end{proof}

\begin{remark}
The reader might be familiar with the fact that one model for the homotopy colimit of a sequence of spaces $X_0\rightarrow X_1\rightarrow X_2\rightarrow ...$ is the \emph{mapping telescope}. Additionally, it is not hard to prove that, in the case where all of the maps in the sequence are closed inclusions, the mapping telescope is homotopy equivalent to the colimit of the sequence. This coincides with the results given above. Although this holds for compactly generated (non-equivariant) spaces, one can extend this result to the equivariant setting.
\end{remark}

\begin{corollary}
\label{lc26}
Let $K$ be a compact $G$-space and let
\[
X_0 \rightarrow X_1 \rightarrow ...\rightarrow X_n \rightarrow ..
\]
be a sequence of $\Ghauss$ spaces where the maps $X_i \rightarrow X_{i+1}$ are closed $G$-inclusions. Then the map
\[
\underset{n\in\Nat}{\colim}\, \Map(K, X_n) \rightarrow \Map( K, \underset{n\in\Nat}{\colim}\, X_n)
\]
is a weak homotopy equivalence of $G$-spaces.
\end{corollary}

\begin{corollary}
\label{hmap}
Let $K$ be a compact $\GCW$-complex and let
\[
X_0 \rightarrow X_1 \rightarrow ...\rightarrow X_n \rightarrow ..
\]
be a sequence of $\Ghauss$ spaces. Then the map
\[
\underset{n\in\Nat}{\hocolim}\, \Map(K,X_n)\rightarrow \Map(K, \underset{n\in\Nat}{\hocolim}\, X_n)
\]
is a weak equivalence of $G$-spaces.
\end{corollary}
\begin{proof}
The homotopy colimit is computed by taking the colimit of the cofibrant replacement of the initial diagram with respect to the projective model structure. In our case, that would include replacing the sequence by a weakly equivalent one involving $G$-cofibrations. The functor $\Map(K,\_)$ with $K$ a compact $\GCW$-complex preserves weak homotopy equivalences of $G$-spaces, hence we can assume that the maps $X_n \rightarrow X_{n+1}$ are $G$-cofibrations, and therefore closed inclusions by \ref{gincl}. Then each map $\Map(K,X_n) \rightarrow \Map(K,X_{n+1})$ is also a closed inclusion. The desired result follows from \ref{lc25} and \ref{lc26}.
\end{proof}
\begin{proposition}
\label{hocopi}
Let
\[
X_0 \rightarrow X_1 \rightarrow ...\rightarrow X_n \rightarrow ..
\]
be a sequence of $\Ghauss$ spaces and let  $H\leq G$ be a closed subgroup. Then the map
\[
\underset{n\in\Nat}{\colim}\, \pi^H_k(X_n) \rightarrow \pi^H_k(\underset{n\in\Nat}{\hocolim}\, X_n)
\]
is an isomorphism.
\end{proposition}

\begin{proof}
As in the previous proof, the homotopy colimit is computed by replacing the initial sequence by a weakly equivalent one involving $G$-cofibrations, which are, in particular, closed inclusions. Let us denote with $\{CX_n\}_{n\in\mathbb{N}}$ the cofibrant replacement diagram of $\{X_n\}_{n\in\mathbb{N}}$. The maps $CX_n\rightarrow X_n$ are all weak homotopy equivalences of $G$-spaces, hence by \ref{lc24} we have have an isomorphism
\begin{align*}
\underset{n\in\Nat}{\colim}\, \pi^H_k(X_n) & \cong\underset{n\in\Nat}{\colim}\, \pi^H_k(CX_n) \rightarrow \pi^H_k(\underset{n\in\Nat}{\colim}\, CX_n) \\
 & \cong \pi^H_k(\underset{n\in\Nat}{\colim}\, X_n)
\end{align*}
for every $k\geq 0$. With this the proof is finished.
\end{proof}

\section{Further topics and applications}

Homotopy limits and colimits represent the variants of the notion of the limits and colimits in the homotopy category. Furthermore, the notion of homotopy limits and colimits coincide with the notion of limits and colimits in the underlying $\infty$-category of a model category (see \cite[Chapter 4]{HTT} for more details), hence, their study in the category of $G$-spaces leads to better understanding of the $\infty$-category of $G$-spaces. In particular, the results presented in this paper were used to prove that the $\infty$-category of $G$-spaces is a presentable $\infty$-category in the PhD thesis of the author. In more detail, the author proved that the $\infty$-category of $G$-spaces is generated under filtered colimits by the compact objects, which are further proved to be finite $\GCW$-complexes. To achieve this, the results such as commutativity of homotopy colimits and mapping spaces were used.

Moreover, these constructions and equivalences extend naturally to the study of $G$-spectra. As such, the techniques and results of this paper may serve as valuable tools for further developments in equivariant stable homotopy theory.


\noindent Universoty pf Belgrade, Faculty of Mathematics, Studentski trg 16, 11158 Belgrade, Serbia\\
E-mail: aleksandar.miladinovic@matf.bg.ac.rs\\
ORCID iD: \url{https://orcid.org/0009-0007-3171-3571}


\begin{thebibliography}{10}
\bibitem[DS95]{DS95}
W.G~Dwyer, J.~Spalinski, \emph{Homotopy theories and model categories}, Handbook of algebraic topology, 1995.

\bibitem[IHom]{Gin}
G.~Ginot, \emph{Introduction à l'homotopie}, course notes, available at the web page of the author.

\bibitem[PH03]{Hir}
Philip S.~Hirschhorn, \emph{Model categories and their localizations}, volume 99 of Mathematical Surveys and Monographs, AMS, Providence, RI, 2003.

\bibitem[Hov98]{Hov}
M.~Hovey, \emph{Model Categories}, American Mathematical Society, Providence, RI, 1998.

\bibitem[HTT]{HTT}
J.~Lurie, \emph{Higher Topos Theory}, Princeton University Press, 2009.

\bibitem[MM02]{MMMF}
M.A.~Mandell, J.P.~May, \emph{Equivariant orthogonal spectra and $S$-modules}, Memoirs AMS 755, American Mathematical Society, 2002.

\bibitem[May99]{MCC}
J.P.~May, \emph{Concise course in algebraic topology}, Chicago Lectures in Mathematics, University of Chicago press, 1999.

\bibitem[CBMS]{CBMS}
J.P.~May, \emph{Equivariant Homotopy and Cohomology Theory}, CBMS Regional Conference Series in Mathematics No. 91, AMS, 1996.

\bibitem[McC69]{McC69}
M.C.~McCord, \emph{Classifying spaces and infinite symmetric products}, Trans. Amer. Math. Soc. 146 (1969), 273–298.

\bibitem[Qui67]{Qui67}
D.G.~Quillen, \emph{Homotopical algebra}, SLNM 43, Springer, Berlin, 1967.

\bibitem[Qui69]{Qui69}
D.G.~Quillen, \emph{Rational homotopy theory}, Ann. of Math. 90 (1969), 205-295.

\bibitem[Rezk]{Rezk}
C.~Rezk, \emph{Compactly generated spaces}, notes available at \url{https://rezk.web.illinois.edu/papers.html}

\bibitem[Sch18]{Swe}
S.~Schwede, \emph{Global Homotopy Theory}, preprint available at \textit{arXiv:1802.09382}

\bibitem[Ste67]{Ste67}
N.E.~Steenrod, \emph{A convenient category of topological spaces}, Michigan Math. J. 14 (1967), 133-152.

\end{thebibliography}
\end{document}